\documentclass[a4paper,12pt]{article}

\usepackage{amssymb}
\usepackage{epsfig}
\usepackage{amsfonts}
\usepackage{amsmath}
\usepackage{euscript}
\usepackage{amscd}
\usepackage{amsthm}
\usepackage{theoremref}
\usepackage{enumerate}
\usepackage{array}
\usepackage{mathrsfs}
\usepackage{tikz-cd}
\DeclareMathAlphabet{\mathpzc}{OT1}{pzc}{m}{it}
\DeclareMathAlphabet{\mathpzc}{OT1}{pzc}{m}{it}

\newtheorem{thm}{Theorem}[section]
\newtheorem{lem}[thm]{Lemma}
\newtheorem{prop}[thm]{Proposition} 
\newtheorem{cor}[thm]{Corollary}

\newtheorem{rem}[thm]{Remark}
\newtheorem{ex}[thm]{Example}

\newcommand{\m}{\mathpzc{m}}
\newcommand{\p}{\mathpzc{p}}

\newcommand{\bZ}{\mathbb Z}

\newcommand{\bC}{\mathbb C}

\newcommand{\A}{\mathbb A}

\newcommand{\X}{X_1,\dots,X_m}
\newcommand{\mi}{1~\leqslant~i~\leqslant~m}

\newcommand{\Spec}{\operatorname{Spec}}

\newcommand{\td}{\operatorname{tr.deg}}

\newcommand{\dk}{\operatorname{DK}}
\newcommand{\ml}{\operatorname{ML}}

\title{ On the family of affine threefolds $a(x)y=F(x,z,t)$ }
\author{
	Parnashree Ghosh$^a$, Neena Gupta$^{b}$ and Ananya Pal$^{c}$\\
	{\small{\it $^a$ School of Mathematics,
			Tata Institute of Fundamental Research}}\\
	{\small{ \it Dr. Homi Bhabha Road,
			Mumbai-400005,India}}\\
	{\small{\it $^{b,c}$ Theoretical Statistics and Mathematics  Unit, Indian Statistical Institute,}}\\ 
	{\small{\it 203 B.T.Road, Kolkata-700108, India}}\\
	{\small{\it e-mail : $^a$ ghoshparnashree@gmail.com}}\\
	{\small{\it e-mail : $^{b}$ neenag@isical.ac.in, rnanina@gmail.com}}\\
	{\small {\it e-mail: $^{c}$ palananya1995@gmail.com }}
}
\begin{document}
	\date{}
	\maketitle
	\abstract{In recent decades, linear affine threefolds have enabled researchers to solve some of the challenging problems on affine spaces.
		Koras-Russell threefolds, especially the Russell Cubic over $\bC$ and Asanuma threefolds over a field of positive characteristic, are striking examples of such linear threefolds.
		In this paper, we apply tools from $K$-theory and theory of $\mathbb{G}_a$-actions to linear threefolds of the form $G:=a(X)Y-F(X,Z,T)\in k[X,Y,Z,T]$, over an arbitrary field $k$.
		
		We give some equivalent conditions for $G$ to be a hyperplane (i.e., $k[X,Y,Z,T]/(G)~=~k^{[3]}$) in the following cases: (i) $k$ is a field of characteristic zero (ii) $k$ is an arbitrary field and $a(X)$ has only multiple roots.
		We also establish the Abhyankar-Sathaye Conjecture affirmatively in these cases.
		
	}

	\smallskip
	
	\smallskip
	
	
	\section{Introduction}
	Throughout the paper, $k$ will denote a field of arbitrary characteristic and $\overline{k}$ will denote an algebraic closure of $k$. 
	All rings considered in this paper are commutative with unity. Capital letters like $X,Y,Z,T,U,V, X_1,\ldots,X_n$ etc., will denote indeterminates
	over the respective ground rings or fields. For a ring $R$,  we write $R^{[n]}$  to denote a polynomial ring in $n$ indeterminates over $R$.
	We shall call a polynomial $G \in k[X_1, \dots, X_n]$ linear if it is linear in one of the indeterminates.

	In the past few decades, solutions of some of the central problems in Affine Algebraic Geometry, like the Linearization Problem and the Zariski Cancellation Problem have involved 
	questions of the following type, for some specified linear polynomials $G~\in~k[X_1, ..., X_n]$:
	
	\medskip
	\noindent
	{\bf Question 1.}
	Is the linear polynomial $G$ under consideration a hyperplane (i.e., whether $\frac{k[X_1,\ldots,X_n]}{(G)}=k^{[n-1]}$) ?
	
	\medskip 
	\noindent

	The investigations also involved the problem of determining whether such a linear polynomial is a coordinate, if it is a hyperplane. This is a special case of one of the most challenging problems in Affine Algebraic Geometry, the {\it Epimorphism Problem} (\cite{kr},~\cite{DG}):
	
	\smallskip
	\noindent
	{\bf Question 2}: Let $m,n$ be two positive integers and $\phi: k[X_1,\ldots,X_n] \twoheadrightarrow k[Y_1,\ldots,Y_m]$, a $k$-algebra epimorphism.
	Does it follow that there exists a system of coordinates $\{F_1,\ldots,F_n\}$ of $k[X_1,\ldots,X_n]$ such that $\ker \phi=(F_1,\ldots,F_{n-m})$?
	
	\medskip
	In particular, when $n-m=1$, i.e., when ker$(\phi)$ is a principal ideal generated by some $G$, we have the following version of the Epimorphism Problem.
	
	\smallskip
	\noindent
	{\bf Question 3}: 	Let $k$ be a field and $n$ be a positive integer $\geqslant 2$.
	Let $G \in k[X_1,\ldots,X_n]$ be such that 
	$\frac{k[X_1,\ldots,X_n]}{(G)}=k^{[n-1]}$. Does it follow that $k[X_1,\ldots,X_n]=k[G]^{[n-1]}$?
	
	\medskip
	When $k$ is of positive characteristic, there are counterexamples to Question~$3$ given by B. Segre and M. Nagata (\cite{Se}, \cite{Na}).
	When $n=2$ and ch.$k=0$,
	Abhyankar-Moh (\cite{AM}) and Suzuki (\cite{Suz}) gave an affirmative answer to Question~$3$ ---  this is popularly known as the ``Epimorphism Theorem".
	When $n\geqslant 3$ and ch.$k=0$, the famous {\it Abhyankar-Sathaye Conjecture} asserts an affirmative answer to Question~$3$. 
	Note that when the characteristic of $k$ is positive, though Question~$3$ has a negative answer in general, the question may still be asked for specific forms of $G$. Indeed, many partial affirmative results for Question~$3$ have been proved even when $k$ is of arbitrary characteristic (see  \cite{rp}, \cite{Wright1}, \cite{DaDu1}, \cite{com}, \cite{adv2} etc.).
	A general survey of the Epimorphism Problem can be found in \cite{DG}.	
	
	For $n=3$, an affirmative solution to Question~$3$ was obtained when  $G~\in~k[X_1,X_2,X_3]$ is a linear plane, 
	first by A. Sathaye  (\cite{sp}) in characteristic zero and  later by P. Russell (\cite{rp}) in arbitrary characteristic.
	They also proved that if $B=k^{[2]}$ and the linear plane $G~\in~B[Y]~(=~k^{[3]})$ is of the form $aY+b$, where $a,b \in B$ and $a \neq 0$, then the coordinates $X, Z$ of
	$B$ can be chosen such that $B=k[X,Z]$ with $a \in k[X]$ and $k[X_1,X_2,X_3]=k[X,G]^{[1]}$; i.e., any  linear plane $G$ in $k^{[3]}$ was shown to be of the form $a(X)Y+b(X,Z)$ and $G$ forms a coordinate along with $X$.
	
	An affirmative answer to Question~$3$ would yield a possible generalisation of the Sathaye-Russell Theorem on linear planes.
	Motivated by the above result of Sathaye and Russell on linear planes, researchers started investigating the following question:
	
	\medskip
	\noindent
	{\bf Question 4}: Let 
	\begin{equation}\label{AI}
		G:=a(X)Y-F(X,Z,T) \text{ and }	A=	\frac{k[X,Y,Z,T]}{(a(X)Y - F(X,Z,T))}, 
	\end{equation}
	
	with $\deg_Xa(X)\geqslant1$,	be a domain. Then
	
	\begin{enumerate} 
		\item [\rm (i)] Under what condition $A=k^{[3]}$?
		\item[\rm(ii)] Does $A=k^{[3]} \implies  k[X,Y,Z,T]=k[G]^{[3]}?$
		\item[\rm(iii)]	If so, is $G$ necessarily a coordinate along with $X$ in $k[X,Y,Z,T]$?
	\end{enumerate}

	S. Kaliman, S. V\'en\'ereau, M. Zaidenberg (\cite{kvz}) and S. Maubach (\cite{CDer})  have answered Question~4, when $k=\bC$ and
	M. E. Kahoui, N. Essamaoui and M. Ouali extended their result over fields of characteristic zero (\cite{interpolation}).
	In \cite{com}, the second author had addressed Question~$4$, when $a(X)=X^d; d\geqslant 2$ and ch.$k\geqslant 0$, as a part of her investigations around the Zariski Cancellation Problem. 
	Further results on Question~$4$ have been obtained recently by the authors in \cite[Theorem 5.29]{XXX} in the following form:
	
	\medskip
	\noindent
	{\bf Theorem I.}
	Let $A$ and $G$ be as in \eqref{AI} and let $x$ be the image  of $X$ in $A$.
	Suppose $$
	a(X)=(X-\lambda)^{d}\alpha_1(X)\in \overline{k}[X],\, d\geqslant 2 ,\, \alpha_1(\lambda)\neq 0,
	$$
	where $\lambda$ is a separable element over $k$ and 
	$$
	F(X,Z,T)=f(Z,T)+\prod\limits_{i=1}^{n}p_i(X)h(X,Z,T),
	$$
	where  $p_1,\dots,p_n$ are the distinct prime factors of $a(X)$ in $k[X]$
	and $f\in k^{[2]}\setminus \{0\}, h\in~k^{[3]}$.
	Then the following statements are equivalent:
	\begin{enumerate}
		\item [\rm(i)] $k[X,Y,Z,T]=k[X,G]^{[2]}$.
		
		\item[\rm(ii)]  $k[X,Y,Z,T]=k[G]^{[3]}$.
		
		\item[\rm(iii)] $A=k[x]^{[2]}$.
		
		\item[\rm(iv)] $A=k^{[3]}$.
		
		\item[\rm(v)]  $k[Z,T]=k[f(Z,T)]^{[1]}$.
	\end{enumerate}
	
	\medskip
	\noindent
	%
	In this paper,	we shall prove the following generalisation of Theorem I (\thref{chp2}):
	
	\medskip
	\noindent
	{\bf Theorem A.}
	Let $A$ and $G$ be as in \eqref{AI} and let $x$ be the image  of $X$ in $A$. Suppose 
	$a(X)$ has no simple root in $\overline{k}$. 
	Then the following statements are equivalent:
	\begin{enumerate}
		\item [\rm(i)] $k[X,Y,Z,T]=k[X,G]^{[2]}$.
		
		\item[\rm(ii)]  $k[X,Y,Z,T]=k[G]^{[3]}$.
		
		\item[\rm(iii)] $A=k[x]^{[2]}$.
		
		\item[\rm(iv)] $A=k^{[3]}$.
		
		\item[\rm(v)] For every root $\lambda$ of $a(X)$, $k(\lambda)[Z,T]=k(\lambda)[F(\lambda,Z,T)]^{[1]}$.
	\end{enumerate}
	
	\medskip
	\noindent
	Theorem~A gives affirmative answer to the Epimorphism Problem (Question~3) for the respective family of hyperplanes in arbitrary characteristic.
	
	In \thref{chp2}, we actually prove an extended version of Theorem A with more equivalent statements, some involving the triviality of the {\it Makar-Limanov invariant}, some involving the triviality of the {\it Derksen invariant} of $A$, and some involving the condition that $A$ is an affine fibration over the subring $k[x]$.
	The theorem establishes a connection among the Dolgachev-Weisfeiler Affine Fibration Problem (about the triviality of $\mathbb{A}^n$-fibrations), the Zariski Cancellation Problem and the Epimorphism Problem. 
	Theorem~A also yields a family of counterexamples to the Zariski Cancellation Problem in positive characteristic (\thref{czcp}).
	In a forthcoming paper \cite{zcp}, further discussion on their isomorphic classes have been undertaken which will show that this family of counterexamples to the Zariski Cancellation Problem are distinct from the existing family of counterexamples given by the Asanuma threefolds in \cite{inv} and \cite{com}.
	The theory developed in Theorem A also allows us to understand the non-triviality of a huge class of threefolds. For example, the hypersurface given by
	$$ 
	G_1=X^2(X+1)^2Y - (Z^2+T^3)-X\in {k_1}[X,Y,Z,T],
	$$
	over an arbitrary field $k_1$ and the hypersurface  given by  
	$$
	G_2=(X^p-\lambda^p)Y -(Z^2+T^3)\in k_2[X,Y,Z,T],
	$$
	over a non-perfect field $k_2$ of characteristic $p>0$ for some $\lambda\in \overline{k_2}\setminus k_2$ (where $\overline{k_2}$ is an algebraic closure of $k_2$) and $\lambda^p\in k_2$,
	are not included in the family of hypersurfaces considered in Theorem I but included in Theorem~A  and it will follow that $k_i[X,Y,Z,T]/(G_i)\not\cong k_i^{[3]},\, i=1,2$.
	%
	%
	%
	
	We will also prove the following theorem (\thref{ch0}).

	\medskip
	\noindent
	{\bf Theorem B.}
	Let $A$ and $G$ be as in \eqref{AI} and let $x$ be the image  of $X$ in $A$. Suppose $k$ is a field of characteristic zero. 
	Then the statements (i)-(v) in Theorem~A are equivalent.

	\medskip
	\noindent
	
	Though Theorem~B has already been proved earlier in \cite{kvz}, \cite{CDer} and \cite{interpolation}, we give an alternative  purely algebraic approach without using any results from topology.
	Further, in \thref{ch0}, each of the five statements of Theorem~B has been shown to be equivalent to statements involving stably polynomial algebras and triviality of Makar-Limanov invariant/Derksen invariant.	
	
	\medskip
	
	Observe that from Theorems~A and
	B, Question~4(i) is addressed by the
	equivalence $\rm(iv)\Leftrightarrow\rm(v)$; 
	Question~4(ii) by the
	equivalence of $\rm(iv)\Leftrightarrow \rm(ii)$ and Question~4(iii) by $\rm(i)\Leftrightarrow\rm(ii)\Leftrightarrow \rm(iv)$.
	In particular, the  Abhyankar–Sathaye Conjecture holds affirmatively for the hypersurfaces $G$, considered in Theorems A and B.
	
	Below we give a layout of this paper.
	In Section \ref{pre}, we recall some basic concepts and well known results which will be used subsequently. In Section \ref{PropA}, we investigate a few properties of the ring $A$ as in \eqref{AI}, like factoriality etc.
	In the last section we prove Theorem~A (\thref{chp2}), Theorem~B (\thref{ch0})
	and present a few examples illustrating the various hypotheses of Theorem~A.
	
	\section{Preliminaries}\label{pre}
	
	Throughout the paper, $k$ will denote a field 
	and for a ring $R$, $R^{[n]}$ denotes a polynomial algebra in $n$ indeterminates over $R$.   
	If $R\subseteq B$ are domains then $\td_R(B)$ denotes the transcendence degree of fraction field of $B$ over the fraction field of $R$.
	For a ring $R$, a prime ideal $\p$ of $R$ and an $R$-algebra $B$, 
	$B_{\p}$ denotes the ring $S^{-1}B$, where $S:=R \setminus \p$,
	$k(\p)$ denotes the field $\frac{R_{\p}}{\p R_{\p}}\cong$ Frac$(\frac{R}{\p})$ 
	and $R^*$ denotes the group of all units of $R$.

	We first recall the concept of exponential maps on a $k$-algebra $B$.

	\medskip
	\noindent
	{\bf Definition.}
	Let $B$ be a $k$-algebra and $\phi: B \rightarrow B^{[1]}$ be a $k$-algebra homomorphism. For an indeterminate $U$ over $B$, 
	let $\phi_{U}$ denote the map $\phi: B \rightarrow B[U]$. Then $\phi$ is said to be an {\it exponential map} on $B$, if the following conditions are satisfied:
	
	\begin{itemize}
		\item [\rm (i)] $\epsilon_{0} \phi_{U} =id_{B}$, where $\epsilon_{0}: B[U] \rightarrow B$ is the evaluation map at $U=0$.
		
		\item [\rm (ii)]  $\phi_{V} \phi_{U}=\phi_{U+V}$, where $\phi_{V}:B \rightarrow B[V]$ is extended to a $k$-algebra homomorphism $\phi_{V}:B[U] \rightarrow B[U,V]$, by setting $\phi_{V}(U)=U$.
	\end{itemize}
	The ring of invariants of $\phi$ is a subring of $B$ defined as follows:
	$$
	B^{\phi}:=\left\{b \in B \, | \phi(b)=b  \right\}.
	$$
	If $B^{\phi}\neq B$, then $\phi$ is said to be {\it non-trivial}.
	Let EXP($B$) denote the set of all exponential maps on $B$. The {\it Makar-Limanov invariant} of $B$ is defined to be
	$$
	\ml(B):= \bigcap_{\phi \in \text{EXP}(B)} B^{\phi}
	$$
	and the {\it Derksen invariant} is a subring of $B$ defined as follows
	$$
	\dk(B):=k\left[ b \in B^{\phi} \,|\, \phi \in \text{EXP}(B)\, \text{and}\, \, B^{\phi} \subsetneq B   \right].
	$$
	
	We record below an useful lemma on exponential maps.
	\begin{lem}\thlabel{poly}
		Let $k$ be a field and $B=k^{[n]}$. Then $\dk(B)=B$ for $n\geqslant 2$ and $\ml(B)=k$.
	\end{lem}
	
	Next we quote a technical result from \cite[Thoerems 5.22 and 5.23]{XXX}.
	
	\begin{prop}\thlabel{lin}
		Let $B$ be a $k$-domain defined as follows 
		$$
		B= \frac{k[X,Y,Z,T]}{(X^d\alpha_1(X)Y-F(X,Z,T))},	\text{~where~} d>1 \text{ and } \alpha_1(0)\neq 0. 
		$$
		Suppose that either $\ml(B)=k$ or $\dk(B)=B$.
		Then there exist $Z_1,T_1 \in k[Z,T]$ and $a_0,a_1\in k^{[1]}$ such that $k[Z,T]~=~k[Z_1,T_1]$ and $f(Z,T)~:=~F(0,Z,T)~=~a_0(Z_1)~+~a_1(Z_1)T_1.$
		
		Moreover, if $k[Z,T]/(f)=k^{[1]}$, then $k[Z,T]=k[f]^{[1]}$.
	\end{prop}

	Let $R$ be a Noetherian ring and let $\mathscr{M}(R)$ denote the category of finitely generated $R$-modules and $\mathscr{P}(R)$ the category of finitely generated projective $R$-modules. Let $G_{0}(R)$ and $G_{1}(R)$, respectively denote the {\it Grothendieck group} and the {\it Whitehead group} of the category  $\mathscr{M}(R)$. Let $K_{0}(R)$ and $K_{1}(R)$, respectively denote the {\it Grothendieck group} 
	and the {\it Whitehead group} of the category  $\mathscr{P}(R)$ (cf. \cite{bass}, \cite{bgl}). For $i \geqslant 2$, the definitions of $G_{i}(R)$ and $K_{i}(R)$ 
	can be found in (\cite{sr}, Chapters 4 and 5).
	Now, we quote two important results from \cite{sr}, Proposition 5.16 and Theorem 5.2 respectively.

	\begin{thm}\thlabel{fcom}
		Let $t$ be a regular element of $R$, $\phi: R \rightarrow C$ be a flat ring homomorphism of Noetherian rings and $u=\phi(t)$. Then we get the following commutative diagram:
		\begin{equation*}
			\begin{tikzcd}
				\cdots\arrow[r] &G_{i}(\frac{R}{tR}) \arrow[r]\arrow[d,"\overline{\phi}^{*}"] & G_{i}(R)\arrow[r]\arrow[d,"\phi^{*}"] & G_{i}(R[t^{-1}])\arrow{r}\arrow[d,"(t^{-1}\phi)^{*}"] & G_{i-1}(\frac{R}{tR}) \arrow{r}\arrow[d, "\overline{\phi}^{*}"] &\cdots  \\
				\cdots\arrow[r] &G_{i}(\frac{C}{uC}) \arrow[r] & G_{i}(C)\arrow[r] &G_{i}(C[u^{-1}]) \arrow{r}  & G_{i-1}(\frac{C}{uC}) \arrow{r} &\cdots ,
			\end{tikzcd}
		\end{equation*}
		where $\phi $ induces the vertical maps.
	\end{thm}
	
	\begin{thm}\thlabel{split}
		For  an indeterminate $T$ over $R$, 
		the maps $G_{i}(R) \rightarrow G_{i}(R[T])$, induced by the inclusion $R \hookrightarrow R[T]$, are isomorphisms for every $i \geqslant 0$.
	\end{thm}
	
	\begin{rem}\thlabel{rmk1}
		\rm{ 
			\begin{enumerate}[\rm(i)]
				\item 
				For a regular ring $R$, $G_{i}(R)=K_{i}(R)$, for every $i \geqslant 0$. 
				In particular, 
				\begin{enumerate}
					\item[\rm(a)] $G_0(k[\X])= G_0(k )= K_0(k)= \mathbb{Z}$
					\item[\rm(b)] $G_{1}(k[\X])=G_1(k)= K_1(k)=k^{*}$
				\end{enumerate}
				
				\item
				There is a canonical group homomorphism $\theta: R^{*}\hookrightarrow G_1(R)$ defined by $\theta(u)=[R,\sigma_{u}]$, for $u\in R^{*}$, where $\sigma_{u}:R\rightarrow R$ is the $R$-linear automorphism defined by $\sigma_{u}(r)=ru$ for all $r\in R$.			 
			\end{enumerate}
		}
	\end{rem}

	We now state the cancellative property of $k^{[1]}$ (\cite{AEH}, (2.8)).
	\begin{thm}\thlabel{aeh}
		Let $R$ be a $k$-domain such that $R^{[n]}=k^{[n+1]}$. Then $R = k^{[1]}$.
	\end{thm}
	
	
	The following is a well-known characterisation of $k^{[1]}$ over an algebraically closed field~$k$.
	
	\begin{lem}{\thlabel{alg}}
		Let $k$ be an algebraically closed field and $B$ a finitely generated $k$-algebra. Suppose that $B$ is a PID and $B^{*}= k^{*}$. Then $B=k^{[1]}$.
	\end{lem}
	
	We now state the well known Epimorphism Theorem due to Abhyankar-Moh (\cite{AM}) and Suzuki  (\cite{Suz}).
	
	\begin{thm}\thlabel{ams}
		Let $k$ be field of characteristic zero and $f \in k[Z,T]$. If  $\frac{k[Z,T]}{(f)}=k^{[1]}$, then $k[Z,T]=k[f]^{[1]}$.
	\end{thm}
	There exist examples constructed by Segre and Nagata which shows that the above theorem does not hold over a field of positive characteristic.

	We now state the celebrated Quillen-Suslin Theorem on projective modules (\cite{su2}, \cite{qu}).
	\begin{thm}\thlabel{qs}
		Every finitely generated projective module over $k^{[n]}$ is free. 
	\end{thm}

	The following theorem on the structure of locally polynomial algebras was proved by Bass, Connell and Wright in \cite{bcw} and independently by Suslin in \cite{su}.
	
	\begin{thm}\thlabel{bcw}
		Let $R$ be a ring, $B$ a finitely presented $R$-algebra and $n\in \mathbb{Z}_{>0}$. Suppose that, for each maximal ideal $\m$ of $R$, the $R_{\m}$-algebra $B_{\m}$ is isomorphic to $R_{\m}^{[n]}$. Then $B \cong Sym_{R}(M)$ for some finitely generated projective $R$-module $M$ of rank $n$.
	\end{thm}

	We now state a version of the Russell-Sathaye criterion \cite[Theorem 2.3.1]{rs} for a ring to be a polynomial ring over a certain subring, as presented in \cite[Theorem 2.6]{BD}.
	
	\begin{thm}\thlabel{rs}
		Let $R \subseteq C$ be integral domains such that $C$ is a finitely generated $R$-algebra. 
		Let $S$ be a multiplicatively closed subset of $R\setminus\{0\}$ generated by some prime elements of $R$ 
		which remain prime in $C$. 
		
		Suppose $S^{-1}C=(S^{-1}R)^{[1]}$ and, for every prime element $p \in S$, we have $pR=pC \cap R$ and $\frac{R}{(p)}$ is algebraically closed in 
		$\frac{C}{(p)}$. 
		
		Then $C=R^{[1]}$.
	\end{thm}
	
	Next we state an epimorphism result of Bhatwadekar and Dutta (\cite[Proposition 3.4 and Theorem 3.5]{BD}), a partial extension of the Sathaye-Russell Theorem over fields to the case of DVRs.
	\begin{thm}\thlabel{bd}
		Let $(R,\pi)$ be a discrete valuation ring, $\kappa:= \frac{R}{(\pi)}$ and $K:= R\left[\frac{1}{\pi}\right]$. Suppose $G \in R[Y,Z,T]$ is of the form $G=aY-b$, where $a \neq 0$ and $a,b \in R[Z,T]$. For any $g \in R[Y,Z,T]$, set $\overline{g}$ to be the image of $g$ in $\frac{R[Y,Z,T]}{(\pi)}$. 
		
		If $\frac{R[Y,Z,T]}{(G)}=R^{[2]}$,
		then there exists an element $Z_0 \in R[Z,T]$ such that $a \in R[Z_0]$, $\overline{Z_0} \notin \kappa$ and $K[Z,T]=K[Z_0]^{[1]}$. 
		
		Moreover, if $\dim (k[\overline{G},\overline{Z_0}])=2$, then $R[Y,Z,T]=R[G]^{[2]}$. In particular, if $a\notin~\pi R[Z,T]$, then $R[Y,Z,T] = R[G]^{[2]}$.
	\end{thm}
	
	The next result on triviality of separable $\A^{1}$-forms over $k^{[1]}$ is a special case of a theorem of Dutta \cite[Theorem 7]{dutta}.
	
	\begin{lem}\thlabel{sepco}
		Let $f \in k[Z,T]$ be such that $L[Z,T]=L[f]^{[1]}$, for some separable field extension $L$ of $k$. Then $k[Z,T]=k[f]^{[1]}$. 
	\end{lem}
	
	We now define an $\A^n$-fibration over a ring $R$.
	
	\medskip
	\noindent
	{\bf Definition.}
	Let $R$ be a ring. 
	A finitely generated flat $R$-algebra $B$ is said to be an {\it $\A^n$-fibration} over $R$ if $B \otimes_R k(\p) = k(\p)^{[n]}$ for every prime ideal $\p$ of $R$.
	
	\medskip
	
	The next result states that over a PID $R$ containing a field $k$ of arbitrary characteristic, an $\mathbb{A}^2$-form which is also an $\mathbb{A}^2$-fibration is trivial (\cite[Theorem 2.8]{dvr}).

	\begin{thm}\thlabel{kx2}
		Let $R$ be a PID containing a field $k$. If $B$ is an $\A^2$-fibration over $R$ such that $B\otimes_k \overline{k} = (R\otimes_k \overline{k})^{[2]}$, then $B=R^{[2]}$.
	\end{thm}
	Note that the above result was proved earlier by A. K. Dutta (\cite[Remark 8]{dutta}) when $k$ is a field of characteristic zero without the assumption of $B$ being an $\mathbb{A}^2$-fibration over~$R$.

	We now recall a result proved by the second author from \cite[Proposition 3.6]{adv}. 
	
	\begin{prop}\thlabel{p1}
		Let $R$ be an integral domain, $\pi_1, \pi_2,\ldots, \pi_n \in R$ and $\pi=\pi_1 \pi_2\cdots \pi_n$. Let $G(Z,T) \in R[Z,T]$ be such that $R[Z,T]/(\pi, G(Z,T)) \cong_R (R/\pi)^{[1]}$. Let $r_1,\ldots, r_n$ be a set of positive integers and
		$$D := R[Z, T,Y ]/(\pi_1^{r_1}\cdots \pi_n^{r_n}Y-G(Z,T)).$$
		Then $D^{[1]}=R^{[3]}$.
	\end{prop}

	Finally, we quote an easy lemma.
	For a proof one can look at \cite[Lemma 5.28]{XXX}.
	\begin{lem}\thlabel{linear}
		Let $f= a_0(Z) + a_1(Z)T$ for some $a_0,a_1\in k^{[1]}$, be an irreducible polynomial of $k[Z,T]$ with $\left( \frac{k[Z,T]}{(f)} \right)^{*}=k^{*}$. Then $k[Z,T] = k[f]^{[1]}$. In particular, if $\frac{k[Z,T]}{(f)} = k^{[1]}$, then $k[Z,T] = k[f]^{[1]}$.	
	\end{lem}

	\section{ Some properties of the ring $A$}\label{PropA}
	
	Throughout this section, $A$ will denote a ring of the following form
	\begin{equation}\label{A}
		A:=	\frac{k[X,Y,Z,T]}{(a(X)Y - F(X,Z,T))}, \text{~where~}\deg_Xa(X)\geqslant1
	\end{equation}
	and $G:= a(X)Y-F(X,Z,T)$. 
	Further, let $x,y,z,t$ denote the images  of $X,Y,Z,T$ in $A$ respectively. 
	Without loss of generality, we further assume that $\deg_X F < \deg_X a(X)$.
	
	Note that, $A$ is an integral domain if and only if $\gcd(a(X), F(X,Z,T))=1$ in $k[X,Z,T]$. We first note down an observation in the form of a lemma.
	
	\begin{lem}\thlabel{gcd} Let $a(X), F(X,Z,T)\in k[X,Z,T]$. Then $
		\gcd(a(X), F(X,Z,T)) = 1$ in $k[X,Z,T]$ if and only if $\gcd(a(X), F(X,Z,T)) = 1$ in $\overline{k}[X,Z,T]$.
	\end{lem} 
	\begin{proof}
		It is enough to show that if ${\rm gcd}(a(X), F(X,Z,T)) \neq 1$ in ${\overline{k}}[X,Z,T]$
		then \\${\rm gcd}(a(X), F(X,Z,T))\neq 1$ in $k[X,Z,T]$. 
		
		Let $\lambda$ be a root of $a(X)$ in $\overline{k}$ such that $F(\lambda, Z,T)=0$ and  
		$p(X)$ be a minimal polynomial of $\lambda$ over $k$.
		Let
		$$
		F(X,Z,T)= \sum_{(i, j) \in \Lambda} \alpha_{ij}(X) Z^iT^j,
		\text{ for some } \alpha_{ij}(X) \in k[X]\setminus \{0\} 
		$$ 
		and some finite subset $\Lambda$ of ${\bZ^2_{\geqslant 0}}$.
		Since $F(\lambda,Z,T)=0$, $\alpha_{ij}(\lambda)=0$, for every $(i,j) \in \Lambda$, and hence $p(X) \mid \alpha_{ij}(X)$ in $k[X]$ for every $(i,j)\in \Lambda$. 
		Therefore, $p(X) \mid F(X,Z,T)$ in $k[X,Z,T]$ and hence  
		${\rm gcd}_{k[X,Z,T]}(a(X), F(X,Z,T)) \neq 1$.
	\end{proof}
	
	As a consequence of \thref{gcd}, it follows that  
	\begin{lem}\thlabel{rnew}
		$A$ is an integral domain if and only if 	$ A \otimes_k \overline{k}$ is an integral domain.
	\end{lem}
	
	For the rest of this section we shall assume that $A$ is an affine domain.
	We now state a criterion for a simple birational extension of a UFD to be a UFD (\cite[Proposition~3.5]{XXX}).
	
	\begin{prop}\thlabel{ufdg}
		Let $R$ be a UFD, $u,v\in R\setminus\{0\}$ and $C=\frac{R[Y]}{(uY-v)}$ be an integral domain.
		We consider $R$ as a subring of $C$.
		Let $u= \prod_{i=1}^{n}u_i^{r_i}$ be a prime factorization of $u$ in $R$. Suppose that for every $i \in \{ 1, \dots, n\}$ for which $(u_i,v)R$ is a proper ideal, we have $\prod_{j\neq i}^{}u_j^{s_j}\notin (u_i,v)R,$ for arbitrary integers $s_j\geqslant 0$. Then the following statements are equivalent:
		\begin{itemize}
			\item [\rm (i)] $C$ is a UFD.
			\item  [\rm (ii)] For each $i$, $1 \leqslant i \leqslant n$, either $u_i$ is a prime element in $C$ or $u_i \in C^*$.
			\item  [\rm (iii)] For each $i$, $1 \leqslant i \leqslant n$, either $(u_i,v)R$ is a prime ideal of $R$ or $(u_i,v)R = R$.
		\end{itemize}  
	\end{prop}
	We now prove some necessary and sufficient conditions for $A$ to be a UFD.
	
	\begin{prop}\thlabel{ufd}
		The following statements are equivalent:
		
		\begin{itemize}
			\item[\rm(i)] $A$ is a UFD.
			
			\item [\rm (ii)] Every prime factor $p(x)$ of $a(x)$ in $k[x]$ is either a prime element or a unit in $A$.
			
			\item[\rm (iii)] For every root $\lambda$ of $a(X)$ in $\overline{k}$, $F(\lambda,Z,T)$ is either irreducible or a unit in $k(\lambda)[Z,T]$.
		\end{itemize}
	\end{prop}
	\begin{proof}  
		Let $a(X) = \prod_{i=1}^{m}p_i(X)^{r_i}$ be a prime factorization of $a(X)$ in $k[X]$. Taking $R=k[X,Z,T]$, $u = a(X)$, $u_i=p_i(X),\, 1\leqslant i \leqslant m$ and $v=F(X,Z,T)$ in \thref{ufdg}, we have $A=\frac{R[Y]}{(uY-v)}$. 
		
		Fix an $i\in \{1,\dots,m\}$.
		Let $R_i:= \frac{R}{u_iR} = \frac{k[X,Z,T]}{(p_i(X))}$
		and $x_i$ be the image of $X$ in $R_i$.
		Since $\prod_{j \neq i} p_j(x_i)^{s_j}$ is a unit in $R_i$ for all $s_j\geqslant 0$,  $\prod_{j \neq i} p_j(x_i)^{s_j} \notin F(x_i,Z,T)R_i$ whenever $F(x_i,Z,T)R_i$ is a proper ideal.
		So, $\prod_{j \neq i} p_j(X)^{s_j} \notin (p_i(X), F(X,Z,T))k[X,Z,T]$, whenever $(u_i, v)R =(p_i(X), F(X,Z,T))k[X,Z,T]$ is a proper ideal. 
		Hence the result follows from \thref{ufdg}.
	\end{proof}
	Using the above proposition, we now prove the following result.
	\begin{cor}\thlabel{corline}
		Let $A$
		be a UFD such that for a root $\lambda$ of $a(X)$ and a minimal polynomial $p(X)$ of $\lambda$ in $k[X]$, $\left(\frac{A}{p(x)A}\right)^{*} =k(\lambda)^{*}$ and  $F(\lambda,Z,T)= a_0(Z)+ a_1(Z)T$ for some $a_0, a_1\in~k(\lambda)^{[1]}$. Then $k(\lambda)[Z,T] = k(\lambda)[F(\lambda,Z,T)]^{[1]}$.
	\end{cor}
	\begin{proof}
		Note that $\frac{A}{p(x)A} \cong  \frac{k(\lambda)[Y,Z,T]}{(F(\lambda,Z,T))}=\left(\frac{k(\lambda)[Z,T]}{(F(\lambda,Z,T))}\right)^{[1]}$.
		Therefore, $\left(\frac{k(\lambda)[Z,T]}{(F(\lambda,Z,T))}\right)^{*}=\left(\frac{A}{p(x)A}\right)^{*}~=~k(\lambda)^{*}$ and since $A$ is a UFD, by \thref{ufd}, we know that $F(\lambda,Z,T)$ is irreducible in $k(\lambda)[Z,T]$.
		Hence, by \thref{linear}, $k(\lambda)[Z,T] = k(\lambda)[F(\lambda,Z,T)]^{[1]}$.
	\end{proof}
	
	We now give necessary and sufficient conditions on $A$ to be an affine fibration over~$k[x]$.

	\begin{prop}\thlabel{fib}
		The following  statements are equivalent:
		\begin{itemize}
			\item [\rm (i)] $A$ is an $\mathbb{A}^2$-fibration over $k[x]$.
			
			\item[\rm (ii)] For every prime factor $p(x)$ of $a(x)$ in $k[x]$, $\frac{A}{p(x)A} = \left( \frac{k[x]}{(p(x))} \right)^{[2]}$.
			
			\item[\rm (iii)] For every root $\lambda$ of $a(X)$ in $\overline{k}$, $F(\lambda, Z,T)$ is a line in $k(\lambda)[Z,T]$, i.e., $ \frac{k(\lambda)[Z,T]}{(F(\lambda,Z,T))}~=~k(\lambda)^{[1]}$.
		\end{itemize}
	\end{prop}
	
	\begin{proof}
		$\rm (i) \Rightarrow (ii):$
		Since $A$ is an $\mathbb{A}^{2}$-fibration over $k[x]$, we have
		$$
		A\otimes_{k[x]} \frac{k[x]_{\p}}{\p k[x]_{\p}} = \left(\frac{k[x]_{\p}}{\p k[x]_{\p}}
		\right)^{[2]},\text{ for every } \p \in \Spec(k[x]).$$
		Hence for an irreducible factor $p(x)$ of $a(x)$ in $k[x]$, we get $\frac{A}{p(x)A}=\left( \frac{k[x]}{(p(x))}\right)^{[2]}.$
		
		\smallskip
		\noindent
		$\rm (ii) \Rightarrow (iii):$
		Let $\lambda$ be a root of $a(X)$ in $\overline{k}$. Therefore, there exists a prime factor $p(X)$ of $a(X)$ in $k[X]$, such that $p(\lambda)=0$. Now $ \frac{k[x]}{(p(x))}\cong k(\lambda)$ and hence, we have
		\begin{equation}\label{ap1}
			\frac{A}{p(x)A}	\cong \left(\frac{k(\lambda)[Z,T]}{(F(\lambda,Z,T))}\right)^{[1]}.		
		\end{equation}
		By $\rm(ii)$, 
		\begin{equation}\label{ap2}
			\frac{A}{p(x)A} = \left( \frac{k[x]}{(p(x))} \right)^{[2]}= k(\lambda)^{[2]}.
		\end{equation}
		Therefore, by (\ref{ap1}), (\ref{ap2}) and \thref{aeh}, we have 
		$\frac{k(\lambda)[Z,T]}{(F(\lambda,Z,T))}= k(\lambda)^{[1]}$.
		
		\smallskip
		\noindent
		$\rm (iii) \Rightarrow (i):$
		Since $A$ is an integral domain, it is a torsion free $k[x]$-module. As $k[x]$ is a PID, it follows that $A$ is a flat $k[x]$-algebra.
		Let $\p:=(p(x))$ be a prime ideal of $k[x]$. 
		It is enough to show that $A \otimes_{k[x]} k(\p)= k(\p)^{[2]}$.
		We now consider two cases and show that in each case $A$ is an $\mathbb{A}^2$-fibration over $k[x]$.
		
		\smallskip
		\noindent
		{\it Case} 1:
		Suppose $p(x) \nmid a(x)$.
		Then $a(x)$ becomes a unit in $A_{(p(x))}$ and hence
		$A_{(p(x))}=k[x]_{(p(x))}[z,t]=k[x]_{(p(x))}^{[2]}$, and therefore 
		\begin{equation*}\label{lf1}
			\frac{A_{(p(x))}}{p(x) A_{(p(x))}}= A \otimes_{k[x]} \frac{k[x]_{(p(x))}}{p(x) k[x]_{(p(x))}}= \left(\frac{k[x]_{(p(x))}}{p(x) k[x]_{(p(x))}}\right)^{[2]}.
		\end{equation*}
		
		\smallskip  
		\noindent
		{\it Case} 2:
		Suppose $p(x) \mid a(x)$.
		Let $\lambda$ be a root of $p(X)$ in $\overline{k}$.
		Since $F(\lambda,Z,T)$ is a line in $k(\lambda)[Z,T]$, we have
		$$
		\frac{A}{p(x)A} 
		\cong \frac{(k[X]/(p(X)))[Y,Z,T]}{(F(X,Z,T))} \cong  \frac{k(\lambda)[Y,Z,T]}{(F(\lambda,Z,T))}=k(\lambda)^{[2]}.
		$$
		Hence, $A \otimes_{k[x]} k(\p)= k(\p)^{[2]}$, where  $k(\p)= \frac{k[x]_{(p(x))}}{p(x) k[x]_{(p(x))}} \cong k(\lambda)$.
	\end{proof}
	Next, we give a necessary and sufficient condition on $A$ to be a regular domain.
	For any $g\in k[X_1,\dots,X_n]$, $g_{X_i}$ denotes $\frac{\partial g}{\partial X_i}$.
	
	\begin{lem}\thlabel{reg}
		Let $k$ be a perfect field. Then $A$ is a regular domain, if and  only if the following conditions are satisfied:
		\begin{enumerate}
			\item [\rm(i)] For every simple root $\lambda$ of $a(X)$ in $\overline{k}$, $k(\lambda)[Z,T]/F(\lambda, Z,T)$ is a regular ring.
			\item [\rm(ii)] For every multiple root $\lambda$ of $a(X)$ in $\overline{k}$, 
			$$
			(F(\lambda,Z,T), F_X(\lambda,Z,T), F_Z(\lambda,Z,T),F_T(\lambda,Z,T))k(\lambda)[Z,T]=k(\lambda)[Z,T].
			$$
		\end{enumerate}
	\end{lem}
	\begin{proof}
		Since $k$ is a perfect field, by the Jacobian criterion \cite[Theorem 30.5]{matr}, $A$ is a regular domain if and only if 	
		\begin{equation*}
			(a(X)Y-F(X,Z,T), a_X(X)Y-F_X(X,Z,T), a(X), F_Z(X,Z,T),F_T(X,Z,T))=k[X,Y,Z,T].
		\end{equation*}
		Thus $A$ is a regular domain if and only if for any root $\lambda$ of $a(X)$, we have		\begin{equation}\label{j}
			( F(\lambda,Z,T), a_X(\lambda) Y-F_X(\lambda,Z,T),F_Z(\lambda,Z,T),F_T(\lambda,Z,T) )=k(\lambda)[Y,Z,T].	
		\end{equation}				
		If $\lambda$ is a simple root of $a(X)$ then $a_X(\lambda)\neq 0$, and hence \eqref{j} is equivalent to
		$$
		\left(  F(\lambda,Z,T), F_Z(\lambda,Z,T),F_T(\lambda,Z,T)\right)=k(\lambda)[Z,T],
		$$
		i.e. $k(\lambda)[Z,T]/F(\lambda, Z,T)$ is a regular ring.
		Therefore, (i) holds.
		
		Now, if $\lambda$ is not a simple root then $a_X(\lambda) = 0$, and hence \eqref{j} is equivalent to 
		$$
		\left(  F(\lambda,Z,T), F_X(\lambda,Z,T), F_Z(\lambda,Z,T),F_T(\lambda,Z,T)\right)=k(\lambda)[Z,T].
		$$
		Therefore, (ii) holds and the result follows.		
	\end{proof}
	
	\section{Main Theorems}\label{THA}
	
	In this section, we prove extended versions of Theorems~A and B over an affine domain $A$ as in \eqref{A}.
	We begin by proving some results which are crucial steps to the theorems.
	
	The following result gives certain necessary conditions on $A$, for $A$ to be stably isomorphic to a polynomial ring over $k$.
	
	\begin{thm}\thlabel{line}
		Let $A$ be an affine domain as in \eqref{A} such that $A^{[l]}=k^{[l+3]}$ for some $l \geqslant 0$.
		Suppose $a(X)= \prod_{1 \leqslant i \leqslant n} (X-\lambda_i)^{\mu_i}$ in $\overline{k}[X]$ with $\mu_i \geqslant 1$ and $f_i:=F(\lambda_i, Z,T)$, for $1\leqslant i\leqslant n$.
		Then the following statements hold: 
		\begin{enumerate}	
			\item[\rm(i)]	For each $i$, $1\leqslant i\leqslant n$,  
			$
			G_{j}(k(\lambda_i)) \cong G_j\left( \frac{k(\lambda_i)[Z,T]}{(f_i)}\right),
			$
			for $j=0,1$, and hence $\left( \frac{k(\lambda_i)[Z,T]}{(f_i)}\right)^{*}= k(\lambda_i)^{*}$.
			\item[\rm(ii)]  If $\mu_i=1$ for some $i$, then $\frac{\overline{k}[Z,T]}{(f_i)}=\overline{k}^{[1]}$.	
			\item[\rm(iii)] If $\mu_i>1$ for some $i$ and if either  $\ml(A)~=~k$ or $\dk(A)~=~A$, then  $k(\lambda_i)[Z,T]=~~k(\lambda_i)[f_i]^{[1]}$.
		\end{enumerate}			
	\end{thm}
	\begin{proof}
		$\rm (i)$	Consider the inclusions
		\begin{tikzcd}
			k \arrow[r, hook, "\beta"] & k[x] \arrow[r,hook, "\gamma"] &A.
		\end{tikzcd} 
		Note that, by \thref{split}, for every $j\geqslant 0$, $\beta$ induces an isomorphism 
		\begin{tikzcd}
			G_{j}(k) \arrow[r, "G_{j}(\beta)"', "\cong" ] & G_{j}(k[x])
		\end{tikzcd}
		and since $A^{[l]}=k^{[l+3]}$, for some $l\geqslant 0$, the inclusion $\gamma\beta$ also induces an isomorphism
		\begin{tikzcd}
			G_{j}(k) \arrow[r, "G_{j}(\gamma\beta)"', "\cong"] & G_{j}(A).
		\end{tikzcd}
		Therefore, $\gamma$ induces an isomorphism 
		$$\begin{tikzcd}
			G_{j}(k[x]) \arrow[r, "G_{j}(\gamma)"', "\cong"] &G_{j}(A),
		\end{tikzcd} \text{ for every }j \geqslant 0.$$
		Let $a(X)=\prod_{i=1}^{m}p_i(X)^{r_i}$ be a prime factorisation of $a(X)$ in $k[X]$ and $$u(x):= \prod_{1 \leqslant i \leqslant m}p_i(x).$$
		Note that $A[u(x)^{-1}]=k[x,u(x)^{-1},z,t]= k[x, u(x)^{-1}]^{[2]}$. 
		Therefore, by \thref{split}, the inclusion $k[x,u(x)^{-1}] \hookrightarrow A[u(x)^{-1}]$ induces an isomorphism $$G_{j}(k[x,u(x)^{-1}]) \xrightarrow{\cong} G_{j}(A[u(x)^{-1}]), \text{ for every } j \geqslant 0.$$ 
		Let
		$\overline{\gamma}: \frac{k[x]}{(u(x))} \rightarrow ~\frac{A}{(u(x))}$			
		be the canonical map induced by the inclusion
		$
		\gamma: k[x] \hookrightarrow A.
		$
		Now by \thref{fcom}, the flat morphism $\gamma:k[x] \hookrightarrow A$ induces the following commutative diagram for every $j \geqslant 1$:  
		$$
		\begin{tikzcd}[column sep=small]
			G_{j}(k[x]) \arrow{r} \arrow[d, "G_j(\gamma)"',"\cong"] & G_{j}(k[x,u(x)^{-1}])\arrow{r}\arrow[d,"\cong"] & G_{j-1}(\frac{k[x]}{(u(x))}) \arrow{r}\arrow{d} & G_{j-1}(k[x])\arrow{r}\arrow[d,"G_{j-1}(\gamma)"',"\cong"] & G_{j-1}(k[x,u(x)^{-1}]) \arrow[d,"\cong"] \\
			G_{j}(A) \arrow{r} & G_{j}(A[u(x)^{-1}])\arrow{r} & G_{j-1}(\frac{A}{(u(x))}) \arrow{r} & G_{j-1}(A) \arrow{r} & G_{j-1}(A[u(x)^{-1}]).
		\end{tikzcd}
		$$
		From the above diagram, applying the Five Lemma, we obtain that
		$\overline{\gamma}: \frac{k[x]}{(u(x))} \rightarrow ~\frac{A}{(u(x))}$,			
		induces an isomorphism of group
		\begin{equation}\label{4}
			G_{j}\left(\frac{k[x]}{(u(x))}\right)\xrightarrow[G_{j}(\overline{\gamma})]{\cong} G_{j}\left(\frac{A}{(u(x))}\right), \text{~for every~} j \geqslant 0.
		\end{equation}
		By the Chinese Remainder Theorem, $$\frac{k[x]}{(u(x))}= \prod_{i=1}^{ m}\frac{k[x]}{(p_i(x))}$$ and $$\frac{A}{(u(x))}=\prod_{i=1}^{ m} \frac{A}{(p_i(x))}=\prod_{i=1}^{ m}\frac{(k[x]/(p_i(x)))[Y,Z,T]}{(F(x,Z,T))}.
		$$ 
		Hence $\overline{\gamma}$ induces
		$\gamma_i: \frac{k[x]}{(p_i(x))} \rightarrow \frac{A}{(p_i(x))}$ for each $i,\, 1 \leqslant i \leqslant m$,
		$$
		G_j\left(\frac{k[x]}{(u(x))}\right)=\bigoplus_{1 \leqslant i \leqslant  m}G_j\left(\frac{k[x]}{(p_i(x))}\right)$$
		and 
		$$
		G_j\left(\frac{A}{(u(x))}\right)=\bigoplus_{1 \leqslant i \leqslant  m}G_j\left(\frac{A}{(p_i(x))}\right)=\bigoplus_{1 \leqslant i \leqslant m} G_j\left(\frac{(k[x]/(p_i(x)))[Y,Z,T]}{(F(x,Z,T))}\right),\text{ for } j=0,1.
		$$ 
		Moreover, for $j=0,1$,  $G_j(\overline{\gamma})=\prod_{i=1}^{m}G_j(\gamma_i)$, 
		where 
		\begin{equation}\label{5}
			G_j(\gamma_i) : G_{j}\left(\frac{k[x]}{(p_i(x))}\right) 
			\rightarrow
			G_{j}\left(\frac{A}{(p_i(x))}\right)=
			G_j\left(\frac{(k[x]/(p_i(x)))[Y,Z,T]}{(F(x,Z,T))}\right),
		\end{equation}
		is the canonical map induced by $\gamma_i$, for all $i,\mi$.
		Now by \eqref{4}, $G_j(\overline{\gamma})$ are isomorphisms, therefore, $G_j(\gamma_i)$'s are also isomorphisms for $j=0,1$ and for all $i,\,\mi$. 
		
		Let $\lambda_i$ be a root of $a(X)$, for some $i$, $1\leqslant i \leqslant n$.
		Then there exists $l\in \{1,\dots,m\}$ such that $p_{l}(\lambda_i)=0$.
		Therefore, $\frac{k[X]}{(p_{l}(X))} \cong k(\lambda_i)$. Hence from \eqref{5} and \thref{split}, we get that,
		$$
		G_{j}(k(\lambda_i))
		\xrightarrow[G_j{(\gamma_i)}]{\cong} G_j\left( \frac{k(\lambda_i)[Y,Z,T]}{(f_i)}\right) \cong G_j\left( \frac{k(\lambda_i)[Z,T]}{(f_i)}\right),
		\text{ for } j=0,1.$$
		Now, $G_1(k(\lambda_i))=k(\lambda_i)^{*}$ and the canonical isomorphism $G_1(\gamma_i): G_1(k(\lambda_i)) \rightarrow G_1 \left( \frac{k(\lambda_i)[Z,T]}{(f_i)} \right)$ maps $k(\lambda_i)^{*}$ to $\left(  \frac{k(\lambda_i)[Z,T]}{(f_i)}\right)^{*}$.
		Hence, 
		it follows that $\left( \frac{k(\lambda_i)[Z,T]}{(f_i)}\right)^{*}= k(\lambda_i)^{*}$.
		
		\medskip
		\noindent
		(ii)
		Let $i\in \{1,\dots,n\}$ be such that $\mu_i=1$.
		Set $\overline{A}=A \otimes_k \overline{k}$.
		Since ${A}^{[l]}= {k}^{[l+3]}$, for some $l\geqslant 0$, we have $\overline{A}^*=\overline{k}^*$ and $\overline{A}$ is a UFD. 
		Therefore, by \thref{ufd},  $f_i$ is irreducible in $\overline{k}[Z,T]$.
		Now, $\overline{A}$ is also a regular domain.
		Therefore,
		by \thref{reg}(i), $\frac{\overline{k}[Z,T]}{(f_i)}$ is a regular domain.
		Next by (i),
		$G_0\left(\frac{\overline{k}[Z,T]}{(f_i)}\right)= G_0(\overline{k})=\mathbb{Z}$.
		Hence, it follows that the ideal class group of $\frac{\overline{k}[Z,T]}{(f_i)}$, ${\rm Cl}\left(\frac{\overline{k}[Z,T]}{(f_i)}\right)=0$. Therefore, $\frac{\overline{k}[Z,T]}{(f_i)}$ is a PID. Again by (i), we have $\left(\frac{\overline{k}[Z,T]}{(f_i)}\right)^{*}=\overline{k}^{*}$.
		Hence by \thref{alg}, $\frac{\overline{k}[Z,T]}{(f_i)}=\overline{k}^{[1]}$.
		
		\medskip
		\noindent
		$\rm (iii)$
		We fix an $i\in\{ 1,\dots,n\}$ such that $\mu_i>1$. 
		Let $$A_i:= A \otimes_k k(\lambda_i)\cong\dfrac{k(\lambda_i)[X,Y,Z,T]}{(X^{\mu_i}\alpha_1(X)Y-F(X+\lambda_i,Z,T))}, \text{ where } \alpha_1(X)=\frac{a(X+\lambda_i)}{X^{\mu_i}},$$ 
		and $\alpha_1(0)\neq 0$.
		Now, $\ml(A_i)=k(\lambda_i)$ or $\dk(A_i)=A_i$ according as $\ml(A)=k$ or $\dk(A)=A$. 
		Therefore, by \thref{lin}, there exist $Z_{1},T_{1} \in k(\lambda_i)[Z,T]$ and $a_0,a_1\in k(\lambda_i){[Z_{1}]}$ such that $k(\lambda_i)[Z_{1},T_{1}]=k(\lambda_i)[Z,T]$ and $f_i=a_0(Z_{1})+a_1(Z_{1})T_{1}$.
		Since $A^{[l]}=k^{[l+3]}$, we have $A$ is a UFD. 
		By $\rm (i)$, $\left(\frac{k(\lambda_i)[Z,T]}{(f_i)}\right)^{*}=k(\lambda_i)^{*}$.
		Hence by \thref{corline},
		$k(\lambda_i)[Z,T]=k(\lambda_i)[f_i]^{[1]}$.
	\end{proof}
	
	When characteristic of $k$ is zero, the above result has the following consequence.	
	\begin{cor}\thlabel{r3}
		Let $k$ be a field of characteristic zero and $A^{[l]}=k^{[l+3]}$ for some $l \geqslant 0$. Then for every root $\lambda$ of $a(X)$, $k(\lambda)[Z,T]=k(\lambda)[F(\lambda,Z,T)]^{[1]}$ if either $\ml(A)=k$ or $\dk(A)=A$.
	\end{cor}
	\begin{proof}
		Suppose $\ml(A)=k$ or $\dk(A)=A$.	
		Set 
		$\overline{A}=A \otimes_k \overline{k}$.
		Then $\overline{A}^{[l]}= \overline{k}^{[l+3]}$ with $\ml(\overline{A})=\overline{k}$ or $\dk(\overline{A})=\overline{A}$. 
		For a root $\lambda$ of $a(X)$, let $ f_{\lambda}= F(\lambda,Z,T)$.
		Then by \thref{line}(ii), 
		$\frac{\overline{k}[Z,T]}{(f_{\lambda})}=\overline{k}^{[1]}$, when $\lambda$ is a simple root and by \thref{line}(iii), $k(\lambda)[Z,T]=k(\lambda)[f_\lambda]^{[1]}$, when $\lambda$ is a multiple root. Now as
		ch.$k=0$, by \thref{ams} we have $\overline{k}[Z,T]=\overline{k}[f_{\lambda}]^{[1]}$, and by \thref{sepco}, $k(\lambda)[Z,T]=k(\lambda)[f_{\lambda}]^{[1]}$.
	\end{proof}
	
	Recall that $G= a(X)Y-F(X,Z,T)$ in \eqref{A}.
	The next result gives a sufficient condition for $\{X,G\}$ to be part of a system of coordinates in $k[X,Y,Z,T]$, over an algebraically closed field $k$.
	
	\begin{lem}\thlabel{G}
		Let $A$ and $G$ be as in \eqref{A}. Suppose that $a(X)= \prod_{1 \leqslant i \leqslant n} (X-\lambda_i)^{\mu_i}$, $\lambda_i \in k$, $\mu_i \geqslant 1$  and $f_i:=F(\lambda_i, Z,T)$, for every $i$. 
		Moreover, we also assume that, for every $i$, $k[Z,T]=k[f_i]^{[1]}$.
		Then $k[X,Y,Z,T]=k[X,G]^{[2]}=k[G]^{[3]}$.
	\end{lem}
	\begin{proof}
		By Theorems \ref{bcw} and \ref{qs}, it is enough to show that for every prime ideal $\p$ of $k[X,G]$, $k[X,Y,Z,T]_{\p} =k[X,G]_{\p}^{[2]}$. We now show that by considering two cases:
		
		\smallskip
		\noindent
		{\it Case} 1:
		Suppose $\p \in \Spec(k[X,G])$ be such that $a(X)\notin\p$. Then clearly
		$k[X,Y,Z,T]_{\p}~=~k[X,G]_{\p}^{[2]}$, since the image of $G$ in $k[X,Y,Z,T]_{\p}$, is almost monic and linear in $Y$.

		\smallskip
		\noindent
		{\it Case} 2:
		Suppose $\p \in \Spec(k[X,G])$ be such that $a(X)\in \p$.
		Then $\p \cap k[X]=(X-\lambda_i)k[X]$ for some $i$, $1 \leqslant i \leqslant n$.				
		Since $k[Z,T]=k[f_i]^{[1]}$, there exists $h_i \in k[Z,T]$ such that $k[Z,T]=k[f_i,h_i]$. 
		Let $p_i:=X-\lambda_i\in k[X]$, 
		$$
		R_i:=k[X]_{(X-\lambda_i)}, ~C_i:=R_i[G,h_i](={R_i}^{[2]}) \text{~~and~~} D_i:=R_i[Y,Z,T]=R_i[Y,f_i,h_i].
		$$ 
		Note that, $ \frac{C_i}{(p_i)}=\frac{R_i}{(p_i)}[f_i,h_i]=\left(\frac{R_i}{(p_i)}\right)^{[2]}$, $\frac{D_i}{(p_i)}=\left( \frac{C_i}{(p_i)}\right)^{[1]}$ and $D_i[p_i^{-1}]=R_i[p_i^{-1}][Y,Z,T]~=~C_i[p_i^{-1}]^{[1]}$. Therefore, by \thref{rs}, $D_i=C_i^{[1]}$,
		and hence $k[X,Y,Z,T]_{\p}=k[X,G]_{\p}^{[2]}$.
	\end{proof}
	
	The next result provides a sufficient condition for $A$ to be a polynomial ring over~$k[x]$.			
	\begin{lem}\thlabel{corG}
		Suppose that for every root $\lambda$ of $a(X)$ in $\overline{k}$, $k(\lambda)[Z,T] = k(\lambda)[F(\lambda,Z,T)]^{[1]}$ then  $A=k[x]^{[2]}$.	
	\end{lem}
	\begin{proof}
		Since for every root $\lambda$ of $a(X)$ in $\overline{k}$, $F(\lambda,Z,T)$ is a coordinate in $k(\lambda)[Z,T]$,  by \thref{fib}, we have, ${A}$ is an $\mathbb{A}^2$-fibration over ${k}[x]$
		and by \thref{G}, we have, $\overline{k}[X,Y,Z,T]=\overline{k}[X,G]^{[2]}$, which implies that $A\otimes_k \overline{k}=\overline{k}[x]^{[2]}$.
		Hence, by \thref{kx2}, $A=k[x]^{[2]}$.	
	\end{proof}
	
	The next result shows that, if $A$ is a polynomial ring over $k[x]$, then $G$ is a coordinate in $k[X,Y,Z,T]$ along with $X$.
	
	\begin{lem}\thlabel{k[x]}
		Let $A = k[x]^{[2]}$. Then $k[X,Y,Z,T] = k[X,G]^{[2]}$.	
	\end{lem}
	\begin{proof}
		By Theorems \ref{bcw} and \ref{qs}, it is enough to show that for every 
		$\p \in \Spec(k[X])$, $k[X]_{\p}[Y,Z,T]=k[X]_{\p}[G]^{[2]}$.
		Let $(p(X)) \in \Spec(k[X])$ and $R :=k[x]_{(p(x))}$. Since $A~=~k[x]^{[2]}$ we have
		\begin{equation}\label{R21}
			\frac{R[Y,Z,T]}{(a(x)Y-F(x,Z,T))}=R^{[2]}.
		\end{equation}
		We consider two cases:
		
		\smallskip
		\noindent
		{\it Case} 1:
		Suppose $a(x) \notin p(x)k[x]$.
		Then $R[Y,Z,T]=R[G]^{[2]}$.
		
		\smallskip
		\noindent
		{\it Case} 2:
		Suppose that $a(x) \in p(x)k[x]$, i.e., $p(X)\mid a(X)$ in $k[X]$.
		
		\noindent
		Let $\lambda$ be a root of $p(X)$ in $\overline{k}$. 
		Then from \eqref{R21}, it follows that, $\frac{k(\lambda)[Y,Z,T]}{(F(\lambda,Z,T))}=k(\lambda)^{[2]}$.
		Therefore, $F(\lambda,Z,T) \notin k(\lambda)^*$, and hence either $\dim k(\lambda)[Z, F(\lambda, Z,T)]=2$ or \\$\dim k(\lambda)[T, F(\lambda, Z,T)]=2$. Therefore, by \thref{bd}, $R[Y,Z,T]=R[G]^{[2]}$.
	\end{proof}

	Now we prove an extended version of Theorem~A.
	
	\begin{thm}\thlabel{chp2}
		Let $A$ be an affine domain over $k$ as in \eqref{A} and $G= a(X)Y-F(X,Z,T)$ with $\deg_Xa(X)\geqslant 1$. Suppose that
		$a(X)$ has no simple root in $\overline{k}$ i.e. $$a(X)=\prod\limits_{i=1}^{n}(X-\lambda_i)^{\mu_i} \in \overline{k}[X], \text{ for some integer } n\geqslant 1 \text{ and } \mu_i>1,\, 1\leqslant i\leqslant n.$$
		Then the following statements are equivalent: 
		\begin{enumerate}
			\item [\rm(i)] $k[X,Y,Z,T]=k[X,G]^{[2]}$.
			
			\item[\rm(ii)]  $k[X,Y,Z,T]=k[G]^{[3]}$.
			
			\item[\rm(iii)] $A=k[x]^{[2]}$.
			
			\item[\rm(iv)] $A=k^{[3]}$.
			
			\item[\rm(v)] 
			For every root $\lambda$ of $a(X)$, $k(\lambda)[Z,T]=k(\lambda)[F(\lambda,Z,T)]^{[1]}$.
			
			\item[\rm(vi)]  
			For every root $\lambda$ of $a(X)$, $F(\lambda,Z,T)$ is a line in $k(\lambda)[Z,T]$ and $\ml(A)=k$.
			
			\item[\rm(vii)] 	$A^{[l]}=k^{[l+3]}$ for some $l \geqslant 0$ and $\ml(A)=k$.
			
			\item [\rm(viii)] $A$ is an $\mathbb{A}^{2}$-fibration over $k[x]$ and $\ml(A)=k$.	
			
			\item[\rm (ix)]
			$A$ is a UFD, $\ml(A)=k$ and for every prime factor $p(X)$ of $a(X)$ in $k[X]$, $\left( \frac{A}{p(x)A}  \right)^{*}=\left( \frac{k[x]}{(p(x))}\right)^{*}$.
			
			\item[\rm(x)]   
			For every root $\lambda$ of $a(X)$, $F(\lambda,Z,T)$ is a line in $k(\lambda)[Z,T]$ and $\dk(A) = A$.

			\item[\rm(xi)] 	$A^{[l]}=k^{[l+3]}$ for some $l \geqslant 0$ and $\dk(A)=A$.
			
			\item [\rm(xii)] $A$ is an $\mathbb{A}^{2}$-fibration over $k[x]$ and $\dk(A)=A$.
			
			\item[\rm (xiii)] 
			$A$ is a UFD, $\dk(A)=A$ and for every prime factor $p(X)$ of $a(X)$ in $k[X]$, $\left( \frac{A}{p(x)A}  \right)^{*}=\left( \frac{k[x]}{(p(x))}\right)^{*}$.
			
		\end{enumerate}
	\end{thm}
	\begin{proof}
		We are going to prove the above equivalence of statements in the following sequence:
		$$
		\begin{tikzcd}[column sep=small]
			{\rm(i)} \arrow[r,Rightarrow] & {\rm(ii)} \arrow[r, Rightarrow]& {\rm(iv)}\arrow[r, Rightarrow]\arrow[rd, Rightarrow] & {\rm (vii)}\arrow[r, Rightarrow]&{\rm(vi)}\arrow[r, Leftrightarrow]&{\rm(viii)}\arrow[r,Rightarrow]&{\rm(ix)}\arrow[r, Rightarrow]& {\rm(v)}\arrow[r, Rightarrow] &{\rm(iii)}\arrow[r, Rightarrow]& {\rm(i)} \\
			& & & {\rm(xi)}\arrow[r, Rightarrow]& {\rm(x)}\arrow[r, Leftrightarrow]&{\rm(xii)}\arrow[r, Rightarrow]&{\rm(xiii)}\arrow[ur, Rightarrow]
		\end{tikzcd}
		$$
		
		Note that, $\rm (vii) \Rightarrow (vi)$ and 	$\rm (xi) \Rightarrow (x)$ both
		follow from \thref{line}(iii) and that
		$\rm (vi) \Leftrightarrow (viii)$ and $\rm (x)\Leftrightarrow (xii)$ follow from $\rm(i) \Leftrightarrow \rm(iii)$ of \thref{fib}.
		
		\smallskip
		\noindent
		$\rm (viii) \Rightarrow (ix)$ and
		$\rm (xii)\Rightarrow (xiii)$: 
		Let $p(X)$ be a prime factor of $a(X)$ in $k[X]$. Then there exists a root $\lambda_j$ of $a(X)$ in $\overline{k}$, for some $j\in \{1,\dots,n\}$, such that $p(\lambda_j)=0$. 
		Therefore, $\frac{k[X]}{(p(X))}\cong k(\lambda_j)$.
		Since $A$ is an $\mathbb{A}^2$-fibration over $k[x]$, by $\rm(i) \Rightarrow \rm(iii)$ of \thref{fib}, we know that, for every root $\lambda_i$ of $a(X)$ in $\overline{k}$, $\frac{k(\lambda_i)[Z,T]}{(F(\lambda_i,Z,T))}=k(\lambda_i)^{[1]}$. Hence by \thref{ufd}, $A$ is a UFD and
		\begin{equation*}
			\dfrac{A}{p(x)A}\cong \dfrac{k(\lambda_j)[Y,Z,T]}{(F(\lambda_j,Z,T))}= \left(\dfrac{k(\lambda_j)[Z,T]}{(F(\lambda_j,Z,T))}\right)^{[1]}= k(\lambda_j)^{[2]}.
		\end{equation*}
		Thus $\left(\frac{A}{p(x)A}\right)^{*} = k(\lambda_j)^{*}=\left(\frac{k[X]}{(p(X))}\right)^{*}.$ 	
		
		\smallskip
		\noindent
		$\rm (ix) \Rightarrow (v)$ and $\rm (xiii) \Rightarrow (v)$:
		Let us fix a root $\lambda_i$ of $a(X)$ for some $i\in \{ 1,\dots,n\}$. 
		Now, $\ml(A \otimes_k k(\lambda_i))=k(\lambda_i)$ or $\dk(A \otimes_k k(\lambda_i))= A \otimes_k k(\lambda_i)$ according as $\ml(A)=k$ or $\dk(A)=A$.
		Since $\mu_i>1$, and 
		$$A \otimes_k k(\lambda_i)\cong\dfrac{k(\lambda_i)[X,Y,Z,T]}{(X^{\mu_i}\alpha_1(X)Y-F(X+\lambda_i,Z,T))}, \text{ where } \alpha_1(X)=\frac{a(X+\lambda_i)}{X^{\mu_i}},$$
		by \thref{lin}, without loss of generality, we can assume that $$F(\lambda_i,Z,T)=a_0(Z)+a_1(Z)T \text{ for some } a_0,a_1 \in k(\lambda_i)^{[1]}.$$
		Note that, 
		there exists a prime factor $p(X)$ of $a(X)$ in $k[X]$ such that $p(\lambda_i)=0$, and hence $\frac{k[X]}{(p(X))}\cong k(\lambda_i).$
		Hence by \thref{corline},
		$k(\lambda_i)[Z,T] = k(\lambda_i)[F(\lambda_i,Z,T)]^{[1]}$. 
		
		Now, $\rm (v) \Rightarrow (iii)$ and $\rm (iii) \Rightarrow (i)$ follow from Lemmas \ref{corG} and \ref{k[x]} respectively and the rest of the equivalences follow trivially.
	\end{proof}
	
	The next result connects the above theorem with the Zariski Cancellation Problem~(ZCP). 
	
	\begin{cor}\thlabel{czcp}
		Let $k$ be a field of positive characteristic and $A$ be an affine domain as in \eqref{A} over $k$ such that $a(X)$ has only multiple roots in $\overline{k}$
		i.e. $$a(X)=\prod\limits_{i=1}^{n}(X-\lambda_i)^{\mu_i} \in \overline{k}[X], \text{ for some integer } n\geqslant 1 \text{ and } \mu_i>1,\, 1\leqslant i\leqslant n$$ and
		the following conditions are satisfied:
		\begin{enumerate}[\rm(a)]
			\item $\frac{k(\lambda_i)[Z,T]}{(F(\lambda_i,Z,T))}=k(\lambda_i)^{[1]}$, for all $i\in \{1,\dots,n\}$	
			\item $k(\lambda_1)[Z,T]\neq k(\lambda_1)[F(\lambda_1,Z,T)]^{[1]}$.
		\end{enumerate} 
		Then $A$ is a counterexample to the ZCP and to the $\mathbb{A}^2$-fibration problem over $k^{[1]}$ in positive characteristic.
	\end{cor}
	\begin{proof}
		Let $a(X)=\prod_{i=1}^{m}p_i(X)^{r_i}$ be a prime factorisation of $a(X)$ in $k[X]$.
		Since $\frac{k(\lambda_i)[Z,T]}{(F(\lambda_i,Z,T))}=k(\lambda_i)^{[1]}$, for all $i\in \{1,\dots,n\}$, we have $\frac{k[X,Z,T]}{(p_j(X),F(X,Z,T))}=\left(\frac{k[X]}{(p_j(X))}\right)^{[1]}$, for all $j\in \{1,\dots,m\}$.
		Therefore, by Chinese Remainder Theorem,  $$\frac{k[X,Z,T]}{(\prod_{i=1}^{m} p_i(X),F(X,Z,T))}=\left(\frac{k[X]}{(\prod_{i=1}^{m} p_i(X))}\right)^{[1]}.$$
		Hence, by \thref{p1}, $A^{[1]}= k^{[4]}$ and since
		$k(\lambda_1)[Z,T]\neq k(\lambda_1)[F(\lambda_1,Z,T)]^{[1]}$, by \thref{chp2}($\rm (v) \Leftrightarrow (iv)$), $A\neq k^{[3]}$. Thus, $A$ is a counterexample to the ZCP. 
		
		By condition (b) and \thref{chp2}($\rm (v) \Leftrightarrow (iii)$), $A\not\cong k[x]^{[2]}$ but by condition (a) and \thref{fib} ($\rm (i) \Leftrightarrow (iii)$),
		$A$ is an $\A^2$-fibration over $k[x]$. Thus $A$ is a non-trivial 
		$\A^2$-fibration over $k[x]$.
	\end{proof}
	
	Next we prove an extended version of Theorem~B.
	
	\begin{thm}\thlabel{ch0}
		Let $k$ be a field of characteristic zero. Then the following statements are equivalent:
		
		\begin{enumerate}
			\item [\rm(i)] $k[X,Y,Z,T]=k[X,G]^{[2]}$.
			
			\item[\rm(ii)]  $k[X,Y,Z,T]=k[G]^{[3]}$.
			
			\item[\rm(iii)] $A=k[x]^{[2]}$.
			
			\item[\rm(iv)] $A=k^{[3]}$.

			\item[\rm(v)] For every root $\lambda$ of $a(X)$, $k(\lambda)[Z,T]=k(\lambda)[F(\lambda,Z,T)]^{[1]}$.

			\item [\rm(vi)] $A$ is an $\mathbb{A}^{2}$-fibration over $k[x]$. 

			\item[\rm(vii)] 	$A^{[l]}=k^{[l+3]}$ for some $l \geqslant 0$ and $\ml(A)=k$.

			\item[\rm(viii)] $A^{[l]}=k^{[l+3]}$ for some $l \geqslant 0$ and $\dk(A)=A$.	
		\end{enumerate} 
	\end{thm}
	\begin{proof}
		We are going to prove the above equivalence of statements in the following sequence:
		$$
		\begin{tikzcd}[column sep=small]
			{\rm(i)} \arrow[r, Rightarrow] & {\rm(ii)} \arrow[r, Rightarrow]& {\rm(iv)}\arrow[r, Rightarrow]\arrow[rd, Rightarrow] & {\rm (vii)}\arrow[r, Rightarrow]& {\rm(v)}\arrow[r, Rightarrow] &{\rm(iii)}\arrow[r, Rightarrow]\arrow[d, Leftrightarrow]& {\rm(i)} \\
			& & & {\rm(viii)}\arrow[ur, Rightarrow] & &{\rm(vi)}
		\end{tikzcd}
		$$
		
		Note that $\rm(vii)\Rightarrow \rm(v)$ and $\rm(viii)\Rightarrow\rm(v)$ both follow from \thref{r3}.
		Now $\rm (v) \Rightarrow (iii)$ and $\rm (iii) \Rightarrow (i)$ follow from \thref{corG} and \thref{k[x]} respectively. 
		$\rm(vi)\Rightarrow\rm(iii)$ follows from a theorem of Sathaye (\cite{sp2}), establishing that any $\mathbb{A}^2$-fibration over a PID $R$, containing $\mathbb{Q}$, must be isomorphic to $R^{[2]}$.	
		Other equivalences follow trivially. 
	\end{proof}	
	
	We now present a few examples to illustrate the necessity of various hypotheses in Theorem~A (\thref{chp2}).
	The first example shows that in \thref{chp2}, the statements $\rm  (ix)$ and  $\rm (xiii)$ do not imply the other statements when $a(X)$ has a simple root. 	
	
	\begin{ex}{\em Consider
			\begin{equation*}
				A= \dfrac{k[X,Y,Z,T]}{(XY - (Z^2 + T^3 +1))},
			\end{equation*}
			where $a(X) = X$ and $f(Z,T)= Z^2 + T^3 +1$. 
			Note that, $\frac{\overline k[Z,T]}{(Z^2 + T^3 +1)}\neq \overline k^{[1]}$.
			Let $x$ be the image of $X$ in $A$. }
	\end{ex}
	
	\noindent
	Then $A$ has the following properties.
	\begin{enumerate}
		
		\item[\rm(i)] $\ml(A) =k$, $\dk(A) = A$ (cf. \cite[Remark 4.7(2)]{com}). 
		
		\item[\rm(ii)]  $\left(  \frac{A}{xA}\right)^*= k^{*}$.
		
		\item[\rm(iii)] $A$ is a UFD (cf. \thref{ufd}).
	\end{enumerate}
	However,
	\begin{enumerate}
		\item[\rm(I)]   Since $\frac{k[Z,T]}{(f(Z,T))}\neq k^{[1]}$, $\frac{A}{xA}\neq k^{[2]}$. Hence $A$ is not an $\mathbb{A}^2$-fibration over $k[x]$. 
		\item [\rm(II)] $A^{[l]}\not\cong k^{[l+3]}$ for every $l\geqslant 0$ (cf. \thref{line}(ii)).
	\end{enumerate}
	
	The next example shows that in statement $\rm(v)$ of \thref{chp2}, the hypothesis ``for every root of $a(X)$" cannot be relaxed to ``for all but at most one root"; not even in the special case when $F\in k[Z,T]$.				
	\begin{ex}
		{\em Let $k$ be a non-perfect field of positive characteristic $p>0$. Therefore, there exists $\lambda\in k$ such that $\beta^p = \lambda$ in $\overline{k}$ with $\beta\in \overline{k}\setminus k$. Let
			\begin{center}
				$A = k[X,Y,Z,T]/(X^p(X^p - \lambda)Y - (Z^p + \lambda T^p + T))$
			\end{center}
			with $a(X) = X^p(X^p -\lambda)$, $f(Z,T) = Z^p +\lambda T^p + T\in k[Z,T]$ and $G = a(X)Y -f(Z,T)$.
			By \cite[Remark 4.5]{asa2},  $f(Z,T)$ is a non-trivial $\mathbb{A}^1$-form i.e.,
			$\frac{k[Z,T]}{(f(Z,T))}\neq k^{[1]}$ and
			$\frac{k(\beta)[Z,T]}{(f(Z,T))} = k(\beta)^{[1]}$.
			Hence $\left(\frac{k[Z,T]}{(f(Z,T))}\right)^{*} = k^{*}$.
			Note that $k(\beta)[Z,T] = k(\beta)[f, Z+\beta T]$
			and $k(\beta)[X,Y,Z,T] = k(\beta)[X,G,Y_1,Z_1]$, where $Y_1:= Y + T^p$ and $Z_1:= Z+X^2T -\beta XT +\beta T$. 
		}
	\end{ex}
	\noindent
	Then $A$ has the following properties:
	\begin{enumerate}
		\item[\rm(i)] $A$ is a UFD (cf. \thref{ufd}).
		\item [\rm(ii)] Since $\left(\frac{k[Z,T]}{(f(Z,T))}\right)^{*}=k^{*}$ and $\frac{k[Z,T]}{(f(Z,T))}\neq k^{[1]}$, by \thref{linear} it follows that $f(Z,T)$ is not linear with respect to any system of coordinates in $k[Z,T]$. Therefore, by \thref{lin}, we have $\ml(A)\neq k$ and $\dk(A)\neq A$ respectively.
		Hence $A\neq k^{[3]}$ (cf. \thref{poly}).
	\end{enumerate}	
	
	\bigskip
	
	\noindent
	{\bf Acknowledgements.}
	The authors thank Professor Amartya K. Dutta for going through the draft and giving valuable suggestions.
	The second author acknowledges Department of Science and Technology (DST), India for their INDO-RUSS project (DST/INT/RUS/RSF/P-48/2021).

\end{document}